\documentclass{amsart}
\usepackage[b5paper,margin=1in]{geometry}

\usepackage[T1]{fontenc}
\usepackage{lmodern}
\usepackage[utf8]{inputenc}

\usepackage{mathtools,amssymb,bm,amsthm}

\newcommand*\bQ{\mathbb{Q}} 
\newcommand*\bZ{\mathbb{Z}} 
\newcommand*\cO{\mathcal{O}} 
\newcommand*\OK{\cO_K}
\newcommand*\eqdef{\coloneqq} 
\newcommand*\defeq{\eqqcolon}
\DeclareMathOperator\Tr{Tr} 
\newcommand*\sqrtDelta{{\mathchoice{\textstyle\sqrt{\!\Delta}}{\sqrt{\!\Delta}}{\sqrt{\!\Delta}}{\sqrt{\!\Delta}}}}
\newcommand*\tc{{\mathrm{c}}}

\newcommand*\Pmod[1]{\mathclose{}\ (\operatorname{mod} #1)}
\newcommand*\defined[1]{\textsl{#1}}
\newcommand*\tight[1]{\mkern2mu{#1}\mkern2mu\relax}

\DeclarePairedDelimiter\qfl\lfloor\rfloor
\DeclarePairedDelimiter\qcl\lceil\rceil
\DeclarePairedDelimiter\abs\lvert\rvert
\DeclarePairedDelimiterX\set[2]\lbrace\rbrace{\,#1\mathclose{}:\mathopen{}#2\,}

\usepackage{amsthm}
\newtheorem{theorem}{Theorem}
\newtheorem{lemmaasthm}[theorem]{Lemma}
\newtheorem{propositionasthm}[theorem]{Proposition}
\newtheorem{corollaryasthm}[theorem]{Corollary}
\def\alphenumi{\renewcommand\theenumi{\alph{enumi}}}

\usepackage{hyperref}

\begin{document}

\author{Tom\' a\v s Hejda}
\address{Charles University, Faculty of Mathematics and Physics, Department of Algebra, So\-ko\-lov\-sk\' a~83, 18600 Praha~8, Czechia
\newline\indent
University of Chemistry and Technology, Prague, Department of Mathematics, Studentská~6, 16000 Praha~6, Czechia}
\email{tohecz@gmail.com}

\author{V\' \i t\v ezslav Kala}
\address{Charles University, Faculty of Mathematics and Physics, Department of Algebra, So\-ko\-lov\-sk\' a~83, 18600 Praha~8, Czechia}
\email{vita.kala@gmail.com}
\urladdr{https://sites.google.com/site/vitakala/}

\thanks{We acknowledge support by Czech Science Foundation (GA\v CR) grant 17-04703Y (TH,VK)
 and partial support by Charles University Research Centre program UNCE/SCI/022 (VK)}

\title{Additive structure of totally positive quadratic~integers}

\begin{abstract}
Let $K=\bQ(\sqrt D)$ be a real quadratic field.
We consider the additive semigroup $\cO_K^+(+)$ of totally positive integers in $K$ and determine its generators (indecomposable integers) and relations; they can be nicely described in terms of the periodic continued fraction for $\sqrt D$. 
We also characterize all uniquely decomposable integers in $K$ and estimate their norms.
Using these results, we prove that the semigroup $\cO_K^+(+)$ completely determines the real quadratic field $K$.
\end{abstract}



\maketitle

\thispagestyle{empty}


\section{Introduction}

The additive semigroup of totally positive integers $\OK^+$ in a totally real number field $K$ has long played a fundamental role in algebraic number theory,
even though more attention has perhaps been paid to the multiplicative structure of the ring $\OK$, for example, to its units and unique factorization into primes.
The most prominent purely additive objects are the \defined{indecomposable elements}, i.e., totally positive integers $\alpha\in\OK^+$ that cannot be decomposed into a sum $\alpha=\beta+\gamma$ of totally positive integers $\beta, \gamma\in\OK^+$.
For example, in 1945, Siegel used them (under the name ``extremal elements'') to prove that if $K$ is a number field different from $\bQ$ and $\bQ(\sqrt 5)$, then there is a totally positive integer in $K$ that cannot be written as sum of any number of squares~\cite{siegel_1945_313}.

In the real quadratic case $K=\bQ(\sqrt D)$, indecomposables can be nicely characterized in terms of continued fraction (semi-)convergents to $\sqrt D$ (see~Section~\ref{sec:prelim}); Dress and Scharlau~\cite{dress_scharlau_1982}
proved an upper bound on the norm of each indecomposable $N(\alpha)\leq D$, which was recently refined by Jang, Kim, Tinkov\' a, Voutier, and the second author~\cite{jang_kim_2016,kala_2016_jnt,tinkova_voutier_2019}.
This stands in contrast to the situation of a general totally real field $K$, where it is much harder to describe indecomposables: Brunotte~\cite{brunotte_1983}
proved an upper bound on their norm in terms of the regulator, but otherwise their structure remains quite mysterious. 

The goal of this article is to study the structure of the whole additive semigroup $\OK^+(+)$.
This is an interesting problem in itself, but 
it seems also necessary for certain applications (such as the recent progress in the study of universal quadratic forms and lattices over $K$ by Kim, Blomer, Yatsyna, and the second author~\cite{kim_2000,blomer_kala_2015,blomer_kala_2018,kala_2016_bams,kala_yatsyna_2018,yatsyna_2019}).

In particular, as indecomposable elements are precisely the generators of $\OK^+(+)$, we need to determine the relations between them.
While the description of indecomposables in terms of the continued fraction is fairly straightforward, it is a priori not clear at all if the same will be the
case for relations, as there could be some ``random'' or ``accidental'' ones. 
Perhaps surprisingly, it turns out that this is not the case and that all the relations (given in Theorem~\ref{prop:two-indeco}) can be described quite elegantly.
A key tool in the proof is the fact that each totally positive integer can be uniquely written
 as a $\mathbb Z^+$-linear combination of two consecutive indecomposables. 
We also show in Corollary \ref{thm:pres} that this corresponds to a cancellative presentation of the semigroup $\OK^+(+)$.

One of course cannot hope to have an analogue of unique factorization in the additive setting, but nevertheless, \emph{some} elements can be uniquely decomposed as a sum of indecomposables. In Theorem~\ref{thm:ud} we characterize all such \defined{uniquely decomposable elements} and obtain again a very explicit result depending only on the continued fraction.
As another application, this then yields a direct proof of Theorem~\ref{thm:isom} that $\OK^+(+)$ (viewed as an abstract semigroup) completely determines $D$ and the number field~$K$.
Let us briefly remark that this result can be viewed alongside a number of beautiful results concerning the (im)possibility of reconstructing a number field from some of its invariants, such as the absolute Galois group, Dedekind zeta-function, or Dirichlet $L$-series, see, e.g.,~\cite{gassmann_1926,kubota_1957,neukirch_1969,uchida_1976,cornellissen_etal_2019}.
A natural question to ask is of course whether an analogue of our Theorem~\ref{thm:isom} holds also for totally real number fields of higher degree,
although this may be quite hard, as we lack a good understanding of indecomposable elements.

Finally, we use our results to estimate the norms of totally positive integers, and in particular, analogously to the results that norms of convergents and indecomposables are at most $2D^{1/2}$ and $D$, respectively,
 we show that the norm of a uniquely decomposable elements is at most of the order $D^{3/2}$ (Theorem~\ref{thm:norm}).

Note that structure theory of semigroups is a well-developed subject, e.g., see  
\cite{clifford_preston_1967,grillet_1995},
and the references therein. 
Among the fundamental topics in the area are
classification problem (characterize a given class of semigroups, e.g., by finding their presentations) and isomorphism problem (decide when two presentations give isomorphic semigroups): our Corollary~\ref{thm:pres} and Theorem~\ref{thm:isom} solve these problems for the semigroups $\OK^+(+)$ (which are not finitely generated).
Semigroups also have important applications to theoretical computer science, for example
in relation to the study of formal languages.
Quite similar to our semigroups $\OK^+(+)$, that can be viewed as subsets of $\mathbb Z^2$, are the important
linear and semilinear subsets of $\mathbb Z^n$~\cite[Ch.~5]{ginsburg_1966}). 
Also in number theory, numerical semigroups, i.e., subsemigroups of $\mathbb Z^+(+)$, are commonly studied, e.g.,
\cite{garcia_sanchez_rosales_1999} study their presentations. As a final example, let us mention that additive subsemigroups of $\mathbb Z^n$, 
that have similar geometric properties as $\OK^+$, were recently used to classify a certain class of semifields~\cite{kala_korbelar_2018}.

\section*{Acknowledgments}
We are grateful to the anonymous referee for pointing out that we were using the incorrect notion of semigroup presentation and for several other very useful comments that helped us improve the article.

\section{Preliminaries}\label{sec:prelim}

Throughout the work, we will use the following notation. We fix a squarefree integer $D\geq2$ and
consider the real quadratic field $K=\bQ(\sqrt D)$ and its ring of integers $\OK$;
we know that $\{1,\omega_D\}$ forms an integral basis of $\OK$, where
\[
	\omega_D\eqdef \begin{cases}
		\sqrt D & \text{if }\textstyle D\equiv 2,3\Pmod 4
	,\\
		\frac{1+\sqrt D}{2} & \text{if }D\equiv 1\Pmod 4
	.\end{cases}
\]
By $\Delta$ we denote the discriminant of $K$, i.e., $\Delta=4D$ if $D\equiv 2, 3\Pmod 4$ and $\Delta=D$ otherwise. 
The norm and trace from $K$ to $\bQ$ are denoted by $N$ and $\Tr$, respectively.

An algebraic integer $\alpha\in\OK$ is \defined{totally positive} iff $\alpha>0$ and $\alpha'>0$,
 where $\alpha'$ is the Galois conjugate of $\alpha$, we write this fact as $\alpha\succ0$; for $\alpha,\beta\in\OK$ we denote by $\alpha\succ \beta$ the fact that $\alpha-\beta\succ0$, 
and by $\OK^+$ the set of all totally positive integers.
We say that $\alpha\in\OK^+$ is \defined{indecomposable} iff it can not be written
 as a sum of two totally positive integers or equivalently iff there is no algebraic integer $\beta\in\cO_K^+$ such that $\alpha\succ\beta$.
We say that $\alpha\in\OK^+$ is \defined{uniquely decomposable}
 iff there is a unique way how to express it as a sum of indecomposable elements.

It will be slightly more convenient for us to work with a purely periodic continued fraction, and so
let
$\sigma_D=[\overline{u_0, u_1, \dots, u_{s-1}}]$ be the periodic continued fraction expansion of
\[
	\sigma_D
	\eqdef \omega_D + \qfl{-\omega_D'}
	= \begin{cases}
		\sqrt D + \qfl{\sqrt D} & \text{if }D\equiv 2,3\Pmod 4
	,\\
		\frac{1+\sqrt D}{2} + \qfl[\big]{\frac{-1+\sqrt D}{2}} & \text{if }D\equiv 1\Pmod 4
	\end{cases}
\]
 (with positive integers $u_i$).
We then have that $\omega_D = \bigl[\qcl{u_0/2}, \overline{u_1, \dots, u_{s}}\bigr]$.
It is well known that $u_1, u_2, \dots, u_{s-1}$ is a palindrome
 and that $u_0=u_s$ is even if and only if $D\equiv 2,3\bmod4$,
 hence $\qcl{u_0/2}= (u_s + \Tr(\omega_D))/2$~\cite[\S\,24.II]{perron_1913}.

Denote the convergents to $\omega_D$ by $p_i/q_i\eqdef\bigl[\qcl{u_0/2}, u_1, \dots, u_i\bigr]$ and
recall that the sequences $(p_i)$, $(q_i)$ satisfy the recurrence
\begin{equation}\label{eq:recurrence}
	X_{i+2} = u_{i+2} X_{i+1} + X_i
	\quad\text{for}\quad
	i\geq-1
\end{equation}
 with the initial condition $q_{-1}=0$, $p_{-1}=q_0=1$, and $p_0=\qcl{u_0/2}$~\cite[\S\,1]{perron_1913}.
Denote $\alpha_i\eqdef p_i-q_i\omega_D'$ and $\alpha_{i,r} = \alpha_i+r\alpha_{i+1}$.
Then we have the following classical facts (see, e.g.,~\cite{dress_scharlau_1982}):
\begin{itemize}
\item The sequence $(\alpha_i)$ satisfies the recurrence~\eqref{eq:recurrence}.
\item We have that $\alpha_i\succ 0$ if and only if $i\geq -1$ is odd.
\item The indecomposable elements in $\OK^+$ are $\alpha_{i,r}$ with odd $i\geq-1$ and $0\leq r\leq u_{i+2}-1$,
 together with their conjugates.
\item We have that $\alpha_{i,u_{i+2}} = \alpha_{i+2,0}$.
\item The indecomposables $\alpha_{i,r}$ are increasing with increasing $(i,r)$ (in the lexicographic sense).
\item The indecomposables $\alpha_{i,r}'$ are decreasing with increasing $(i,r)$.
\end{itemize}

We also denote $\varepsilon>1$ the fundamental unit of $\OK$; we have that $\varepsilon=\alpha_{s-1}$.
Furthermore, we denote $\varepsilon^+>1$ the smallest totally positive unit $>1$; we
 have that $\varepsilon^+=\varepsilon$ if $s$ is even and
 $\varepsilon^+=\varepsilon^2=\alpha_{2s-1}$ if $s$ is odd.
Furthermore, we denote $\gamma_0=\omega_D$ and $\gamma_i = [u_i, u_{i+1}, u_{i+2}, \dotsc]$ for $i\geq1$;
 we have that $u_i<\gamma_i = u_i + \frac{1}{\gamma_{i+1}}<u_i+1$ for $i\geq1$.
 
In the final Section~\ref{sec:norm} we will estimate norms of totally positive integers
 and particularly uniquely decomposable integers.
To this end, we will use the following additional notation:
For a convergent $\alpha_i$, we set
\[
	N_i \eqdef \abs{N(\alpha_i)}
	= (-1)^{i+1}N(\alpha_i)
	= \begin{cases}
		\abs{p_i^2-Dq_i^2} & \text{if }D \equiv 2, 3 \Pmod{4}
	,\\
		\abs{p_i^2-p_iq_i-q_i^2\frac{D-1}{4}} & \text{if }D\equiv 1\Pmod 4
	. \end{cases}
\]
Recall that we have $p_{i+1}q_i-p_iq_{i+1}=(-1)^i$ and let us define $T_i$ so that $\alpha_i\alpha_{i+1}'=T_{i+1}+(-1)^i\omega_D$. This means that we have $T_i=p_i(p_{i-1}-q_{i-1})-q_iq_{i-1}\frac{D-1}4$ or
$T_i=p_ip_{i-1}-Dq_iq_{i-1}$ when $D\equiv 1\pmod 4$ or $2, 3\pmod 4$, respectively.

\section{Relations in the semigroup \texorpdfstring{$\OK^+(+)$}{OK+(+)}}

In this section we will consider all the (additive) relations that hold in the semigroup $\OK^+(+)$. We will show in Theorem~\ref{prop:two-indeco} that they all follow from certain basic relations, and that this amounts to giving a presentation of $\OK^+(+)$ as a commutative, cancellative semigroup (Corollary \ref{thm:pres}).

Let $$\mathcal{A}\eqdef \set{ \alpha_{i,r} }{ i\geq -1\text{ odd and } 0\leq r\leq u_{i+2}-1 }\backslash \{1\}$$ denote the set of indecomposable elements $>1$ and
$\mathcal{A}' \eqdef \set{ y' }{ y\in\mathcal{A} }$ the set of their conjugates, that is, of indecomposables $<1$. Note that each totally positive integer is a finite sum of indecomposables
(for, by considering the trace $\Tr\colon K\rightarrow\bQ$ it is easy to see that there are no infinite descending chains in $\OK^+$ with respect to $\succ$); in other words, the indecomposables 
$\mathcal{A}\cup\mathcal{A}'\cup\{1\}$ generate the semigroup $\OK^+(+)$.

\begin{subequations}\label{eq:pres-Rx}
Let us now consider the following relations (in $\OK$) between the indecomposables:
\begin{gather}\label{eq:R1}
	\alpha_{i,r-1} - 2\alpha_{i,r} + \alpha_{i,r+1} = 0
	\quad
	\text{for odd $i\geq -1$ and $1\leq r\leq u_{i+2}-1$},
\\\label{eq:R2}
	\alpha_{i-2,u_i-1} - (u_{i+1}+2)\alpha_{i,0} + \alpha_{i,1} = 0
	\quad
	\text{for odd $i\geq1$},
\\\label{eq:R3}
	\text{same relations as in \eqref{eq:R1} and \eqref{eq:R2} after applying the automorphism $(')$}
,\\\label{eq:R4}
	\alpha_{-1,1}' - (u_0+2)\cdot 1 + \alpha_{-1,1} = 0
.\end{gather}
\end{subequations}

We will shortly see that these relations hold (in fact, it is a straightforward verification from the definitions),
but let us first introduce an alternative notation of the indecomposables for convenience.
We define $\beta_j$, $j\in\bZ$, by the condition that
 $\dotsb < \beta_{-3} < \beta_{-2} < \beta_{-1} < \beta_0=1 < \beta_1 < \beta_2 < \beta_3 < \dotsb$
 is the increasing sequence of the indecomposables.
Note that we have $\beta_j'=\beta_{-j}$ for all $j\in\bZ$.

\begin{lemmaasthm}\label{lem:1v1}
For each $j\in\bZ$ we have that
\[
	v_j \beta_j = \beta_{j-1}+\beta_{j+1}
,\]
 where
\[
	v_j \eqdef \begin{cases}
		2
		& \text{if $\beta_{\abs j} = \alpha_{i,r}$ with odd $i\geq-1$ and $1\leq r\leq u_{i+2}-1$}
	,\\
		u_{i+1}+2
		& \text{if $\beta_{\abs j} = \alpha_{i,0}$ with odd $i\geq-1$}
	.\end{cases}
\]
\end{lemmaasthm}

\begin{proof}
As $\beta_{-j} = \beta_j'$, we can assume $j\geq0$.
We have $\beta_j = \alpha_{i,r}$ for some odd $i\geq-1$ and $0\leq r\leq u_{i+2}-1$.

If $1\leq r\leq u_{i+2}-1$, then $\beta_{j-1}=\alpha_{i,r-1}=\alpha_i+(r-1)\alpha_{i+1}$, $\beta_j=\alpha_{i,r}=\alpha_i+r\alpha_{i+1}$ and $\beta_{j+1}=\alpha_{i,r+1}=\alpha_i+(r+1)\alpha_{i+1}$
 form an arithmetic sequence, whence $\beta_j-\beta_{j-1}=\beta_{j+1}-\beta_j$, which is the statement.

If $r=0$ and $i\geq1$, we have a multiple of $\alpha_{i,0} = \alpha_i$ as the left-hand side
 and $\alpha_{i-2,u_i-1}+\alpha_{i,1}$ as the right-hand side.
We use the definition of $\alpha_{i,r}$ and the recurrence~\eqref{eq:recurrence} for $\alpha_j$ to see that
\begin{multline*}
\qquad
	\alpha_{i-2,u_i-1} + \alpha_{i,1}
	= \alpha_{i-2} + (u_i-1)\alpha_{i-1} + \alpha_i + \alpha_{i+1}
\\
	= \alpha_i - \alpha_{i-1} + \alpha_i + u_{i+1}\alpha_i + \alpha_{i-1}
	= (u_{i+1}+2)\alpha_i
	= (u_{i+1}+2)\alpha_{i,0}
.\qquad
\end{multline*}

Finally, consider $r=0$ and $i=-1$, i.e., the case $j=0$.
We have $\beta_1=\alpha_{-1,1}=\alpha_1+1=\qcl{u_0/2}-\omega'+1$, whence
\[
	\beta_1+\beta_{-1}
	= \Tr(\beta_1)
	= \Tr(\qcl{u_0/2}-\omega_D'+1)
	= 2\qcl{u_0/2} - \Tr(\omega_D) +2 = u_0 +2
.
\qedhere
\]
\end{proof}

Note that with the notation from Lemma~\ref{lem:1v1}, we can rewrite the
relations~\eqref{eq:R1}--\eqref{eq:R4}
in a unified way in terms of $\beta_j$ and $v_j$ as follows:
\begin{equation}\label{eq:pres-R}
	\beta_{j-1} - v_j \beta_j + \beta_{j+1} = 0
	\quad\text{for}\quad
	j\in\bZ
.\end{equation}

As our first main theorem,
we can now show that each totally positive integer is a linear combination of two consecutive indecomposables with non-negative integral coefficients.
A variant of this statement concerning sums of powers of units was used by Kim, Blomer, and the second
 author~\cite{kim_2000,blomer_kala_2018} in the construction of universal quadratic forms.

\begin{theorem}\label{prop:two-indeco}
Let $x\in\OK^+$ be given as a finite sum $x=\sum k_j \beta_j$
 with $k_j\in\bZ$.
Then there exist unique $j_0,e,f\in\bZ$ with $e\geq1$ and $f\geq0$ such that
 $x = e\beta_{j_0} + f\beta_{j_0+1}$.

Every relation of the form $\sum h_j\beta_j=0$ (with $h_j\in\bZ$ and only finitely many non-zeros) is
a $\bZ$-linear combination of the relations~\eqref{eq:pres-R}; in particular, this is true for
$e\beta_{j_0} + f\beta_{j_0+1} - \sum k_j \beta_j = 0$.
\end{theorem}

\begin{proof}
As the sequence $(\beta_j)_{j\in\bZ}$ is strictly increasing from $0$ to $+\infty$
 and the sequence $(\beta_j')_{j\in\bZ}$ is strictly decreasing from $+\infty$ to $0$,
 we know that the sequence $(\beta_j/\beta_j')_{j\in\bZ}$ is strictly increasing from $0$ to $+\infty$.
As $x\succ0$, we have that $x/x'>0$.
Hence we can fix the unique index $j_0$ such that
\begin{equation}\label{eq:two-indeco-frac}
	\frac{\beta_{j_0}}{\beta_{j_0}'} \leq \frac{x}{x'} < \frac{\beta_{j_0+1}}{\beta_{j_0+1}'}
.\end{equation}

We will first show that $x = \sum_{j=j_{\min}}^{j_{\max}} k_j \beta_j$
 can be rewritten to
\begin{equation}\label{eq:two-indeco-ef}
	x = e\beta_{j_0} + f\beta_{j_0+1}
	\quad\text{with}\quad
	e,f\in\bZ
\end{equation}
 using only relations~\eqref{eq:pres-R}.
We assume that $j_{\min}\leq j_0$ and $j_{\max}\geq j_0+1$
 (if not, we pad the sum by zeros to achieve this).
We use induction on the length of the sum $L=j_{\max}-j_{\min}+1\geq2$:
\begin{itemize}

\item Suppose $L=2$.
 Then $j_{\min}=j_0$ and $j_{\max}=j_0+1$.
 Putting $e = k_{j_0}$ and $f = k_{j_0+1}$ gives the statement.

\item Suppose $L\geq3$ and $j_{\min} \leq j_0-1$.
 Then
 \[
	x
	= (k_{j_{\min}+1} + v_{j_{\min}+1}k_{j_{\min}}) \beta_{j_{\min}+1}
	+ (k_{j_{\min}+2} - k_{j_{\min}}) \beta_{j_{\min}+2}
	+ \!\!\!\sum_{j=j_{\min}+3}^{j_{\max}}\!\!\! k_j \beta_j
 \]
  (here we used~\eqref{eq:pres-R} for $j=j_{\min}+1$) is a sum of the same form with length $L-1$.

\item Suppose $L\geq3$ and $j_{\min}=j_0$.
 Then $j_{\max}=j_0+L-1\geq j_0+2$ and
 \[
	x
	= \!\sum_{j=j_{\min}}^{j_{\max}-3}\! k_j \beta_j
	+ (k_{j_{\max}-2} - k_{j_{\max}}) \beta_{j_{\max}-2}
	+ (k_{j_{\max}-1} + v_{j_{\max}-1}k_{j_{\max}}) \beta_{j_{\max}-1}
 \]
  (here we used~\eqref{eq:pres-R} for $j=j_{\max}-1$) is a sum of the same form with length $L-1$.

\end{itemize}
This finishes the proof of
\eqref{eq:two-indeco-ef} (note that so far we have not used the property of $j_0$). 

We now plug~\eqref{eq:two-indeco-ef} into~\eqref{eq:two-indeco-frac} to derive that
\[
	0\leq f(\beta_{j_0+1}\beta_{j_0}'-\beta_{j_0}\beta_{j_0+1}')
	\quad\text{and}\quad
	0< e(\beta_{j_0+1}\beta_{j_0}'-\beta_{j_0}\beta_{j_0+1}')	
.\]
As $\beta_{j_0}'>\beta_{j_0+1}'>0$ and $\beta_{j_0+1}>\beta_{j_0}>0$, we have that
 $\beta_{j_0+1}\beta_{j_0}'-\beta_{j_0}\beta_{j_0+1}'>0$, hence $e>0$ and $f\geq0$.

To prove the uniqueness of $j_0$ from the statement of the proposition, suppose to the contrary that $j_1\neq j_0$ and $x=e_1 \beta_{j_1} + f_1 \beta_{j_1+1}$
 with $e_1>0$ and $f_1\geq0$.
We first consider the case when $j_1\leq j_0-1$, i.e.,
\[
	\frac{\beta_{j_1}}{\beta_{j_1}'}
	< \frac{\beta_{j_1+1}}{\beta_{j_1+1}'}
	\leq \frac{x}{x'}= \frac{e_1\beta_{j_1}+f_1\beta_{j_1+1}}{e_1\beta_{j_1}'+f_1\beta_{j_1+1}'}
.\]
Easy manipulation leads to $e_1(\beta_{j_1+1}\beta_{j_1}'-\beta_{j_1}\beta_{j_1+1}')\leq0$, hence $e_1\leq0$.
Likewise, in the case $j_1\geq j_0+1$ we get $f_1<0$.

Once $j_0$ is fixed, the uniqueness of $e,f$ follows from the linear independence of
 $\beta_{j_0},\beta_{j_0+1}$ over $\bQ$.
 
\bigskip
 
For the second part of the theorem, note that we have already proved this for $e\beta_{j_0} + f\beta_{j_0+1} - \sum k_j \beta_j = 0$
(as we deduced \eqref{eq:two-indeco-ef} only using \eqref{eq:pres-R}).

Consider now any relation between indecomposables $\sum h_j\beta_j = 0$
(where of course only finitely many $h_j$'s are non-zero).
Let us write this as two sums,
\begin{equation}\label{eq:pres-xpm}
\sum_{h_j>0} h_j\beta_j - \sum_{h_j<0} (-h_j)\beta_j = 0
,\quad\text{i.e.,}\quad
x\eqdef\sum_{h_j>0} h_j\beta_j = \sum_{h_j<0} (-h_j)\beta_j
.\end{equation}
Clearly $x\in\OK^+$.
By the first part of the theorem, $x$ is \emph{uniquely} written as $x=e\beta_j+f\beta_{j+1}$ for some $j\in\bZ$,
$e\geq1$, $f\geq0$; and we know that both $e\beta_j+f\beta_{j+1}=\sum_{h_j>0} h_j\beta_j$ and 
$e\beta_j+f\beta_{j+1}=\sum_{h_j<0} (-h_j)\beta_j$ follow from \eqref{eq:pres-R}.
Thus also the identity between the two right-hand sides follows from \eqref{eq:pres-R}.
\end{proof}
In order to show that Theorem \ref{prop:two-indeco} amounts to a certain presentation of the semigroup $\OK^+(+)$, let us first recall some relevant definitions (see \cite{grillet_1995} for more background and for any undefined notions). Note that we will use the additive notation, as we are interested in the additive semigroup $\OK^+(+)$, although in the literature it is perhaps more common to view semigroups multiplicatively.

Let $S(+)$ be a \emph{commutative semigroup}, i.e., a set $S$ together with a binary operation $+$ that is commutative and associative. $S$ is \emph{cancellative} iff 
$x+z=y+z$ implies $x=y$ (for all $x,y,z\in S$). In this case, let $G(S)=S-S$ be the \emph{universal} (or \emph{Grothendieck}) group of $S$, i.e., the group of (formal) differences $a-b$ for $a,b\in S$, where $a-b=c-d$ in $G(S)$ iff $a+d=b+c$ in $S$ (see \cite[Ch.~II, \S\S\,1,~2]{grillet_1995}). As $S$ is cancellative, we can view it as a subset of $G(S)$, i.e.,  $S\subset G(S)$.

Let us now introduce presentations (following, e.g., \cite[Ch.~I]{grillet_1995}). 
Let $X$ be a set and let $\mathcal{R}$ be a set of relations between the elements of $X$, i.e., formal expressions $x_1+\dots+x_u=y_1+\dots+y_v$ for $x_i, y_j\in X$.
Let $F_X(+)$ be the free commutative semigroup over the set of generators $X$ and let $\mathcal C$ be the congruence on $F_X$ generated by the relations $\mathcal{R}$ (see \cite[Ch.~1, Proposition~2.9]{grillet_1995}).
Then $\langle X \mid \mathcal{R}\rangle$ is a \emph{presentation} of a commutative semigroup $S(+)$ iff $S$ is isomorphic to the quotient $F_X/\mathcal C$ of $F_X$ by the congruence $\mathcal C$.
Note that $F_X/\mathcal C$ in general need not be cancellative (although $F_X$ is), and so, when considering cancellative semigroups, it is sometimes more convenient to consider cancellative presentations.

We say that $\langle X \mid \mathcal{R}\rangle_\tc$ is a \emph{cancellative presentation} of a cancellative, commutative semigroup $S(+)$ iff $S$ is isomorphic to the quotient $F_X/\mathcal D$, where $\mathcal D$ is the smallest cancellative congruence on $F_X$ that contains the relations $R$. The universal property of presentations implies that it is also isomorphic to 
the factor of $\langle X \mid \mathcal{R}\rangle$ by the smallest cancellative congruence (cf.~\cite[Ch.~II, Proposition~2.3]{grillet_1995}).

A little less formally, $\langle X \mid \mathcal{R}\rangle_\tc$ is a cancellative presentation a cancellative, commutative semigroup $S(+)$
if $S$ is generated by $X$ and all (additive) relations between elements of $X$ are generated by consecutively applying the relations in $\mathcal{R}$ and by using the cancellative law `$x+z=y+z$ implies $x=y$'.

\begin{corollaryasthm}\label{thm:pres}
	Let $X=\set{B_j}{j\in\bZ}$ and 
	\begin{equation}\label{eq:pres-T}
	\mathcal{T} \colon
	B_{j-1} + B_{j+1} = v_j B_j
	\quad\text{for}\quad
	j\in\bZ
	\end{equation}
	(where $v_j$ are as in Lemma \ref{lem:1v1}).
	Then $\langle X \mid \mathcal{T}\rangle_\tc$ is a cancellative presentation of the cancellative, commutative semigroup $\OK^+(+)$.
\end{corollaryasthm}

\begin{proof}
Let us consider the (semigroup) homomorphism $\varphi\colon F_X\rightarrow\OK^+$ defined by $\varphi(B_j)=\beta_j$.
	
The indecomposables $\beta_j$ are generators of $\OK^+$, and so $\varphi$ is surjective. Let $\mathcal E$ be the corresponding congruence on $F_X$, i.e., $a\mathcal E b$ iff $\varphi(a)=\varphi(b)$.
Hence we have an isomorphism of cancellative semigroups $\OK^+\simeq F_X/\mathcal E$ and $\mathcal E$ is a cancellative congruence.

Let further $\mathcal D$ be the smallest cancellative congruence on $F_X$ that contains the relations $\mathcal T$.
Since the relations (corresponding to) $\mathcal T$ hold for $\beta_j$ by Lemma \ref{lem:1v1}, we have $\mathcal T\subset \mathcal E$. Furthermore, as $\OK^+$ is a cancellative commutative semigroup, we have $\mathcal D\subset \mathcal E$. 
In order to show that $\langle X \mid \mathcal{T}\rangle_\tc$ is indeed cancellative presentation of $\OK^+(+)$, it remains to show that $\mathcal D=\mathcal E$. This is essentially clear from Theorem \ref{prop:two-indeco}, but we can also proceed formally as follows.

As in \cite[Ch.~II, Proposition~5.1]{grillet_1995}, let $R(\mathcal D)$ be the \emph{R\'edei group} of $\mathcal D$, i.e., $R(\mathcal D)$ is the subgroup
of $G(F_X)$ (which is clearly the free commutative group generated by~$X$) such that $R(\mathcal D)=\set{a-b}{a,b\in F_X, a\mathcal D b}$. Similarly let let $R(\mathcal E)$ be the R\'edei group of~$\mathcal E$. The R\'edei group uniquely determines the corresponding cancellative congruence, and so it suffices to show that $R(\mathcal D)=R(\mathcal E)$.

By the minimality of $\mathcal D$, we see that $R(\mathcal D)$ is the smallest subgroup of $G(F_X)$ that contains $B_{j-1}- v_j B_j + B_{j+1}$ for all $j$.
By definition, $R(\mathcal E)$ consists precisely of sums $\sum h_jB_j$ (with $h_j\in\bZ$) such that $\sum h_j\beta_j=0$ holds (in $\OK$). But by the second part of Theorem \ref{prop:two-indeco},
each such relation $\sum h_j\beta_j=0$ follows from the relations \eqref{eq:pres-R}. We have thus proved that $R(\mathcal D)=R(\mathcal E)$.
\end{proof}
Note that it is easy to see that the given set of relations~\eqref{eq:pres-T}  is minimal in the sense that none of them can be removed.

Let us also note that we indeed need to work with cancellative presentations. For example, $B_0+B_3=(v_1-1)B_1+(v_2-1)B_2$ holds in $\langle X \mid \mathcal{T}\rangle_\tc$, but does not hold in
$\langle X \mid \mathcal{T}\rangle$, as we cannot apply any of the relations \eqref{eq:pres-T} to this identity (without using cancellation).

\section{Uniquely decomposable elements}

We now characterize all the uniquely decomposable elements in $\OK^+$ using the results of the previous section.

\begin{theorem}\label{thm:ud}
All uniquely decomposable elements $x\in\OK^+$ are the following:
\begin{enumerate}\alphenumi
\item $\alpha_{i,r}$ with odd $i\geq-1$ and $0\leq r\leq u_{i+2}-1$;
\item $e\alpha_{i,0}$ with odd $i\geq -1$ and with $2\leq e\leq u_{i+1}+1$ 
\item $\alpha_{i,u_{i+2}-1} + f\alpha_{i+2,0}$ with odd $i\geq -1$ odd such that $u_{i+2}\geq2$ and with $1\leq f\leq u_{i+3}$;
\item $e\alpha_{i,0} + \alpha_{i,1}$ with odd $i\geq -1$ such that $u_{i+2}\geq2$ and with $1\leq e\leq u_{i+1}$;
\item $e\alpha_{i,0} + f\alpha_{i+2,0}$ with odd $i\geq -1$ such that $u_{i+2}=1$ and with
 $1\leq e\leq u_{i+1}+1$, $1\leq f\leq u_{i+3}+1$, $(e,f)\neq(u_{i+1}+1,u_{i+3}+1)$;
\item Galois conjugates of all of the above.
\end{enumerate}
\end{theorem}

Note that the first item lists all the indecomposables (which are clearly uniquely decomposable) and the second one comprises small multiples of all convergents $\alpha_i\succ 0$, including positive integers $1,2,\dots, u_0+1$.

\begin{proof}
The statement of the theorem, when rewritten using $\beta_j$ and $v_j$ (defined in Lemma~\ref{lem:1v1}), says that $x$
 is uniquely decomposable if and only if $x=e\beta_j+f\beta_{j+1}$ with
\begin{equation}\label{eq:klv}
	1\leq e\leq v_j-1
,\quad
	0\leq f\leq v_{j+1}-1
\quad\text{and}\quad
	(e,f)\neq  (v_j-1,v_{j+1}-1)
.\end{equation}

Let $x$ be uniquely decomposable. 
By Theorem~\ref{prop:two-indeco}, it can be written as $x=e\beta_j+f\beta_{j+1}$ for some $e,f,j\in\bZ$ with $e\geq 1, f\geq 0$.
Suppose now that $e,f$ do not satisfy~\eqref{eq:klv}.
This can happen only in one of the following ways:
\begin{itemize}

\item $e\geq v_j$.
 Then $x=e\beta_j+f\beta_{j+1} = \beta_{j-1} + (e-v_j)\beta_j + (f+1)\beta_{j+1}$ are two decompositions of $x$ (see Lemma~\ref{lem:1v1}).

\item $f\geq v_{j+1}$.
 Then $x=e\beta_j+f\beta_{j+1} = (e+1)\beta_j + (f-v_{j+1})\beta_{j+1} + \beta_{j+2}$ are two decompositions of $x$.

\item $(e,f) = (v_j-1,v_{j+1}-1)$.
 Then $x=(v_j-1)\beta_j+(v_{j+1}-1)\beta_{j+1} = \beta_{j-1} + \beta_{j+2}$ are two decompositions of $x$.

\end{itemize}

Conversely, let us now show that the condition~\eqref{eq:klv} is sufficient.
First note that~\eqref{eq:klv} implies
\begin{equation}\label{eq:klv-preceq}
	x+\beta_j\preceq v_j\beta_j+v_{j+1}\beta_{j+1}
\quad\text{or}\quad
	x+\beta_{j+1}\preceq v_j\beta_j+v_{j+1}\beta_{j+1}
.\end{equation}
Suppose now that $x=e\beta_j+f\beta_{j+1}$ with $e,f$ satisfying~\eqref{eq:klv} is not uniquely decomposable.
This can happen only in one of the following ways:
\begin{itemize}

\item $x$ has a decomposition that contains some $\beta_i$ 
 with $i\geq j+2$,
  i.e., $x\succeq\beta_i\geq\beta_{j+2}$.
 Simultaneously, we have from~\eqref{eq:klv-preceq} that
  $x+\beta_{j-1}\leq v_j\beta_j+v_{j+1}\beta_{j+1} = \beta_{j-1}+\beta_{j+2}$, whence $x<\beta_{j+2}$.
 This is a contradiction.

\item $x$ has a decomposition that contains some $\beta_i$ 
 with $i\leq j-1$,
  i.e., $x \succeq \beta_i$.
 Then $x' \geq \beta_i' \geq \beta_{j-1}'$.
 We also have from~\eqref{eq:klv-preceq} that
  $x' + \beta_{j+2}' \leq v_j\beta_j'+v_{j+1}\beta_{j+1}' = \beta_{j-1}'+\beta_{j+2}'$, and so $x'<\beta_{j-1}'$,
 a contradiction.

\item $x=e\beta_j+f\beta_{j+1} = e_1\beta_j+f_1\beta_{j+1}$ for $(e,f)\neq(e_1,f_1)$.
 This is impossible as $\beta_j,\beta_{j+1}$ are linearly independent over $\bQ$.
\end{itemize}
\end{proof}

Note that similarly to the classical formula for the number of indecomposables as a sum of some coefficients $u_i$ (e.g.,~\cite{blomer_kala_2015}), we can now express the number of uniquely decomposable elements.

\begin{corollaryasthm}\label{cor:count-ud}
The number of uniquely decomposable elements of $\OK^+$ modulo totally positive units (i.e., powers of $\varepsilon^+$) is equal to
\begin{align*}
	\sum_{i=1}^s u_i
	+ 2\sum_{\substack{i=2\\\mathclap{i\text{ even}}}}^s u_i
	+ \; \sum_{\substack{i=1\\\mathclap{i\text{ odd},\,u_i=1}}}^{s-1} \; u_{i-1}u_{i+1}
	&\quad\text{in case $s$ even}
,\\
	4\sum_{i=1}^s u_i
	+ \sum_{\substack{i=1\\\mathclap{u_i=1}}}^{s} u_{i-1}u_{i+1}
	&\quad\text{in case $s$ odd}
.\end{align*}
\end{corollaryasthm}

\begin{proof}
Denote $s^+=s$ for $s$ even and $s^+=2s$ for $s$ odd.
Then by direct computation from Theorem~\ref{thm:ud}, we obtain that the number is, for each item in the theorem statement:
\begin{gather*}
\text{(a) }\sum u_{i+2};\qquad
\text{(b) }\sum u_{i+1};\qquad
\text{(c) }\sum_{\mathclap{u_{i+2}\geq2}} u_{i+3};\qquad
\text{(d) }\sum_{\mathclap{u_{i+2}\geq2}} u_{i+1};\\
\text{(e) }\sum_{\mathclap{u_{i+2}=1}} \bigl((u_{i+1}+1)(u_{i+3}+1)-1\bigr) 
	= \sum_{\mathclap{u_{i+2}=1}} (u_{i+1}+u_{i+3}) + \sum_{\mathclap{u_{i+2}=1}} u_{i+1}u_{i+3},
\end{gather*}
 where all the sums are over odd $i$ between $1$ and $s^+$,
 because $\alpha_{i+s^+} = \varepsilon^+\alpha_i$ for all odd $i\geq-1$
 so this restriction on $i$ picks one representative from the uniquely decomposable integers
 for each class modulo powers of $\varepsilon^+$.
Rearranging the sums and using that $u_{i+s}=u_i$ we get the result
$	\sum_{i=1}^{s^+} u_i
	+ 2\sum_{\substack{\\i=1\\\mathclap{i\text{ even\vphantom{d}}}}}^{s^+} u_i
	+ \sum_{\substack{\\i=1\\\mathclap{i\text{ odd},\,u_i=1}}}^{s^+-1} u_{i-1}u_{i+1}$;
 this finishes the proof for $s$~even.

For $s$ odd, note that all the summands are perodic with period $s$,
 hence a sum over all $1\leq i\leq s^+$ is twice the sum over all $1\leq i\leq s$
 and a sum over odd $1\leq i\leq s^+$ is equal to the sum over all $1\leq i\leq s$.
\end{proof}
As another application, we use Theorem~\ref{thm:ud} to prove that the additive semigroups of totally positive integers of different real quadratic fields are non-isomorphic.
This is in stark contrast to the situation of the groups $\OK(+)$, which are all isomorphic to $\bZ^2(+)$.

\

\

\begin{theorem}\label{thm:isom}
The additive semigroups $\OK^+$, for real quadratic fields $K$, are pairwise not isomorphic.
\end{theorem}

\begin{proof}
Assume that we are given $\OK^+(+)$ as an abstract semigroup $S(+)$.
To prove the uniqueness of $K=\bQ(\sqrt D)$, we shall reconstruct the continued fraction for $\sigma_D$
 using the additive structure of $S$.
Note that the indecomposable and uniquely decomposable elements are well-defined in $S$
 as they are defined intrinsically just from the semigroup structure.
Since $S$ is a cancellative semigroup, we can consider the universal (or Grothendieck) group of differences $G(S)=S-S$ and have $S\subset G(S)$.

Consider all the indecomposables such that their double is uniquely decomposable:
\[
	A \eqdef \set{ \alpha\in S }{ \text{$\alpha$ indecomposable and $2\alpha$ uniquely decomposable} }
.\]
From Theorem~\ref{thm:ud} we know that $A$ is exactly the set of convergents and their conjugates, $A=\set{\alpha_i, \alpha_i'}{\text{$i$ odd}}$.
For $\alpha\in A$, denote $k_\alpha$ the maximum integer such that $k_\alpha\alpha$ is uniquely decomposable.
From Theorem~\ref{thm:ud}cd we know that for each $\alpha\in A$ there are exactly two $\beta\in S$ such that
\begin{equation}\label{eq:kab}
	k_\alpha \alpha + \beta
	\text{ is uniquely decomposable}
,\end{equation}
namely if $\alpha=\alpha_i$ then $\beta=\alpha_{i,1}=\alpha+\alpha_{i+1}$ or
 $\beta=\alpha_{i-2,u_i-1}=\alpha_{i-2,u_i}-\alpha_{i-1}=\alpha-\alpha_{i-1}$.
 
Consider now an infinite (bipartite) graph $G$ with vertices $A\cup B$, where
\[
	B \eqdef \set[\big]{ \{\alpha-\beta,\beta-\alpha\} }{ \alpha\in A,\, \alpha,\beta\text{ satisfy~\eqref{eq:kab}} }
,\]
 where $\alpha-\beta\in G(S)$.
Note that $B$ actually contains pairs $\{\alpha_{i+1},-\alpha_{i+1}\}$ and $\{\alpha_{i+1}',\allowbreak-\alpha_{i+1}'\}$ with odd $i$ (but we have no intrinsic way of distinguishing $\alpha_{i+1}$ from $-\alpha_{i+1}$).
The edges in $G$ are defined as follows: There is an edge between $\alpha\in A$ and $\{\gamma,-\gamma\}\in B$
 iff $\beta=\alpha-\gamma$ or $\beta=\alpha+\gamma$ satisfy~\eqref{eq:kab}.
Then $G$ is actually an infinite chain corresponding to 
\[
	\dots, \alpha_3',\{\alpha_{2}',-\alpha_{2}'\},\alpha_1',
\{\alpha_{0}',-\alpha_{0}'\},\alpha_{-1}=\alpha_{-1}'=1,\{\alpha_{0},-\alpha_{0}\},\alpha_1,\{\alpha_{2},-\alpha_{2}\},\alpha_3,\dots
\]

We add labels on each vertex of $G$ in the following way:
 We label $\alpha\in A$ by $k_\alpha-1$.
For $\{\gamma,-\gamma\}\in B$ and its neighbors $\alpha,\tilde\alpha$,
 we have that $\alpha-\tilde\alpha=l_\gamma\gamma$ for some $l_\gamma\in\bZ$
 and we label $\{\gamma,-\gamma\}$ by $\abs{l_\gamma}$.
To see the motivation behind this, assume that $\gamma=\alpha_{i+1}$ for some odd $i$;
 then $\alpha-\tilde\alpha=\pm(\alpha_{i+2}-\alpha_i)=\pm(\alpha_i+u_{i+2}\alpha_{i+1}-\alpha_i)=\pm u_{i+2}\alpha_{i+1}$.

Whence we know that for $i$ odd, $\alpha_i$ is labelled by $u_{i+1}$
 and $\{\alpha_{i+1,}-\alpha_{i+1}\}$ is labelled by $u_{i+2}$
 (with the same being true for $\alpha_i'$).
This means that the infinite chain of the labels is equal to
\[
	\dotsc, u_{s-1}, u_s=u_0, u_1, u_2, \dots, u_{s-2}, u_{s-1}, u_s=u_0, u_1, \dotsc
\]
 (where it does not matter in which direction we read the chain as it is a palindrome).
We know that $u_0$ is the maximal value in the chain and that $s$ is the (shortest) period of the chain.
This means that we have reconstructed $\sigma_D = [\overline{u_0; u_1,\dots,u_{s-1}}]$ just from the intrinsic properties of $S(+)$.
\end{proof}
\section{Estimating norms}\label{sec:norm}

We have seen in Theorem~\ref{prop:two-indeco} that (up to conjugation) every element of $\mathcal O_K^+$ can be uniquely written in the form
$e\alpha_{i, r}+f\alpha_{i, r+1}$, where $i\geq -1$ is odd, $0\leq r\leq u_{i+2}-1$, $e\geq 1$, and $f\geq 0$.
We will now estimate the norm of such an element, generalizing the results concerning indecomposables by Dress, Scharlau, Jang, Kim, the second author (and others)~\cite{dress_scharlau_1982,jang_kim_2016,kala_2016_jnt};
here we will use the notation introduced at the end of Section~\ref{sec:prelim}.

Let us start by recalling the following classical fact (for the proof see, e.g.,~\cite[Proposition~5]{kala_2016_jnt} and~\cite[Lemma~3]{blomer_kala_2018}).

\begin{lemmaasthm}\label{lem:old-formulas}
For all $i\geq -1$, we have
\[
	N_{i+1}=\frac{\sqrtDelta}{\gamma_{i+2}}-\frac{N_{i}}{\gamma_{i+2}^2}
,\qquad
	T_{i+1}=(-1)^{i+1}\left(\omega_D-\frac{N_i}{\gamma_{i+2}}\right)
.\]
\end{lemmaasthm}

We are interested in 
\[
	e\alpha_{i, r}+f\alpha_{i, r+1}=(e+f)\alpha_i+\bigl(re+(r\tight+1)f\bigr)\alpha_{i+1}
,\]
and so let us first prove a general formula for the norm of the element on the right-hand side.

\begin{lemmaasthm}\label{lem:general-norm}
Let $m,n\in\bZ$ and $i\geq -1$ odd. Then
\[
	N(m\alpha_i+n\alpha_{i+1})=\biggl(m-\frac n{\gamma_{i+2}}\biggr)\biggl(n\sqrtDelta+mN_i-n\frac{N_i}{\gamma_{i+2}}\biggr)
.\]
\end{lemmaasthm}

\begin{proof}
Let's prove the result only when $D\equiv 2,3\pmod 4$, as the other case is very similar.
Using the definitions and the previous lemma, we compute
\[\begin{multlined}[b]
	\mkern-10mu
	N(m\alpha_i+n\alpha_{i+1})
	=m^2N(\alpha_i)+n^2N(\alpha_{i+1})+mn(\alpha_i\alpha_{i+1}'+\alpha_i'\alpha_{i+1})
\\[0.5ex]
	=m^2N_i-n^2N_{i+1}+2mnT_{i+1}
	=m^2N_i-n^2\biggl(\frac{2\sqrt D}{\gamma_{i+2}}{-}\frac{N_{i}}{\gamma_{i+2}^2}\biggr)+2mn\biggl(\!\sqrt D{-}\frac{N_i}{\gamma_{i+2}}\biggr)
\\
	=\biggl(m-\frac n{\gamma_{i+2}}\biggr)\biggl(2n\sqrt D+mN_i-n\frac{N_i}{\gamma_{i+2}}\biggr)
.
\end{multlined}
\qedhere
\]
\end{proof}
Now we return to the original situation when $m=e+f$ and $n=re+(r\tight+1)f$ and estimate the norm of totally positive integers.

\begin{propositionasthm}\label{prop:estimate}
Consider $\alpha=e\alpha_{i, r}+f\alpha_{i, r+1}\in\OK^+$ with $i\geq -1$ odd, $0\leq r<u_{i+2}$, $e\geq 1$, and $f\geq 0$.
Then we have the following upper bounds on $N(\alpha)$:
\[
	N(\alpha)<\sqrtDelta\bigl((r\tight+1)e+(r\tight+2)f\bigr)(e+f)
\quad\text{and}\quad
	N(\alpha)<(e+f)^2 \frac{\Delta}{4N_{i+1}}
.\]
For a lower bound we distinguish four cases:
\begin{enumerate}
\item If $f>0$ and $r=0$, then $N(\alpha)> ef\sqrtDelta$.

\item Let $c\in(0,1)$. If $f>0$ and $1\leq r\leq c u_{i+2} -1$, then $N(\alpha)> (1-c)(e+f)^2\sqrtDelta$.

\item If $f>0$ and $\frac{u_{i+2}+1}2<r\leq u_{i+2}-1$, then $N(\alpha)>\frac {\sqrtDelta} 2 e(e+f)$.

\item If $f=0$ and $r>0$, then $N(\alpha)>e^2 \bigl(1-\frac 1{u_{i+2}}\bigr) \sqrtDelta$.
\end{enumerate}
\end{propositionasthm}

Note that when $f=r=0$, then $\alpha=e\alpha_i$ is a multiple of a convergent and we have a good control
 on the size of its norm by Lemma~\ref{lem:old-formulas}.
Hence we have not included the lower bound for this case in the proposition.

\begin{proof}
Let $m=e+f$ and $n=re+(r\tight+1)f$.
By the previous proposition, we want to estimate 
\begin{equation}\label{eq:NAB}
	N(m\alpha_i+n\alpha_{i+1})=
		\,\underbrace{\!
			\biggl(m-\frac n{\gamma_{i+2}}\biggr)
		\!}_{\defeq A}
		\,\underbrace{\!
			\biggl(n\sqrtDelta+mN_i-n\frac{N_i}{\gamma_{i+2}}\biggr)
		\!}_{\defeq B}
.\end{equation}
We have $A = e\bigl(1-\frac r{\gamma_{i+2}}\bigr)+f\bigl(1-\frac{r+1}{\gamma_{i+2}}\bigr)$
 and $B=n\sqrtDelta+AN_i$.


We start with the upper bounds. We have $A<e+f$ and $B<n\sqrtDelta+(e+f)\sqrtDelta=\sqrtDelta((r\tight+1)e+(r\tight+2)f)$, where we used the easy consequence of Lemma~\ref{lem:old-formulas} that $N_i<\sqrtDelta$.

For the second estimate we use a slightly different argument (which we again give only in the case $D\equiv 2,3\pmod 4$):
\begin{multline*}
	N(\alpha)
	= N\biggl(\frac{(m\alpha_i+n\alpha_{i+1})\alpha_{i+1}'}{\alpha_{i+1}'}\biggr)
	= \frac{N(m\alpha_i\alpha_{i+1}'+nN(\alpha_{i+1}))}{N(\alpha_{i+1})}
\\
	= \frac{N(mT_{i+1}-nN_{i+1}- m\sqrt D)}{-N_{i+1}}
	= \frac{(mT_{i+1}-nN_{i+1})^2- m^2D}{-N_{i+1}}
\\
	=\frac{m^2D-(mT_{i+1}-nN_{i+1})^2}{N_{i+1}}<\frac{m^2D}{N_{i+1}}
.\end{multline*}


We finish with the lower bounds.
Let us assume first that $f>0$ and distinguish three cases according to the size of $r$.
\begin{itemize}

\item Case $r=0$. Then $A>e$ and $B>\sqrtDelta f$.

\item Case $1\leq r\leq c u_{i+2} -1$. Then $r+1\leq c u_{i+2}<c \gamma_{i+1}$, and so $1-\frac r{\gamma_{i+2}}>1-\frac {r+1}{\gamma_{i+2}}>1-c$. 
Thus $A=e\bigl(1-\frac r{\gamma_{i+2}}\bigr)+f\bigl(1-\frac {r+1}{\gamma_{i+2}}\bigr)>(1-c)(e+f)$. Moreover, $B>\sqrtDelta(e+f)r$.

\item Case $\frac{u_{i+2}+1}2<r\leq u_{i+2}-1$. Again, $B>\sqrtDelta(e+f)r$.
Note that the function $r\mapsto\bigl(1-\frac r{\gamma_{i+2}}\bigr)r$ is increasing for $r>\frac{u_{i+2}+1}2$.
If we apply this observation, we obtain
$Ar>e\bigl(1-\frac r{\gamma_{i+2}}\bigr)r>e\bigl(1-\frac {u_{i+2}-1}{\gamma_{i+2}}\bigr)(u_{i+2}-1)
=e\frac {1+(\gamma_{i+2}-u_{i+2})}{u_{i+2}+(\gamma_{i+2}-u_{i+2})}(u_{i+2}-1)
>e\frac {u_{i+2}-1}{u_{i+2}}>\frac e2$,
 where we use the fact that $u_{i+2}\geq r+1\geq2$.
\end{itemize}

Finally, if $f=0$, then $N(\alpha)=e^2N(\alpha_{i,r})$, so proving the case $e=1$ is sufficient.
If $r=0$, then estimates on the norm of the convergent $\alpha_i$ are well-known, see, e.g.,~\cite[Lemma~5]{blomer_kala_2018}.
Assume hence $1\leq r\leq u_{i+2}-1$.
We have 
\[
	N(\alpha_{i,r})=\biggl(1-\frac r{\gamma_{i+2}}\biggr)\biggl(r\sqrtDelta+N_i-r\frac{N_i}{\gamma_{i+2}}\biggr)
	> \biggl(1-\frac r{\gamma_{i+2}}\biggr) r\sqrtDelta
.\]

Since the minimum of the function $r\mapsto\bigl(1-\frac r{\gamma_{i+2}}\bigr) r$ is at one of the endpoints of the considered interval $1\leq r\leq u_{i+2}-1$, we conclude that $N(\alpha_{i,r})>\bigl(1-\frac 1{u_{i+2}}\bigr) \sqrtDelta$.
\end{proof}
Let us conclude the paper with some applications of the previous proposition.
From the theory of continued fractions it quite easily follows that all elements
 with absolute value of norm less than $\sqrtDelta/4$ are convergents~\cite[Proposition~7]{blomer_kala_2018}.
However, often it is useful to know all elements of norm less than $\sqrtDelta$ (or other small multiples of $\sqrt D$):
 the characterization of all such totally positive integers follows easily from our proposition.

Finally, we prove an upper bound on the norm of uniquely decomposable elements,
 similar to the bounds on norms of convergents $N_{i+1}=|N(\alpha_i)|<\frac{\sqrtDelta}{u_{i+1}}$
 and indecomposables $N(\alpha_{i, r})\leq\frac{\Delta}{4N_{i+1}}$ (for odd $i$).

\begin{theorem}\label{thm:norm}
If $\alpha\in\mathcal O_K^+$ is uniquely decomposable, then
\[\textstyle
	N(\alpha)<\sqrtDelta\bigl(2\sqrtDelta+1\bigr)\bigl(3\sqrtDelta+2\bigr)
.\]
\end{theorem}

\begin{proof}
All uniquely decomposable $\alpha$ are described in Theorem~\ref{thm:ud}, and so we just need to check the estimate on the norm in each of these cases.
This is very easy when $\alpha$ is an indecomposable or a multiple of a convergent. In the rest of the proof we will often use the easy observation that $u_j\leq u_s<\sqrtDelta$ for all $j$.

Case $\alpha=e\alpha_{i} + f\alpha_{i+2}$ with odd $i\geq -1$ such that $u_{i+2}=1$ and with
$1\leq e\leq u_{i+1}+1$, $1\leq f\leq u_{i+3}+1$, $(e,f)\neq(u_{i+1}+1,u_{i+3}+1)$.
In the notation of Proposition~\ref{prop:estimate} we have $r=0$, and so the first upper bound gives
$N(\alpha)<\sqrtDelta (e+2f)(e+f)$. Thus the estimate follows from the range of $e,f$ using $u_j<\sqrtDelta$.

Case $e\alpha_{i,0} + \alpha_{i,1}$ with odd $i\geq -1$ such that $u_{i+2}\geq 2$ and with $1\leq e\leq u_{i+1}$. We have $f=1$ and $r=0$, and so 
we similarly as in the previous case get $N(\alpha)<\sqrtDelta(\sqrtDelta+1)(\sqrtDelta+2)$.

Case $\alpha_{i,u_{i+2}-1} + f\alpha_{i+2,0}$ for $i\geq -1$ odd and $1\leq f\leq u_{i+3}$. Now $e=1$ and $r=u_{i+2}-1$. In the notation~\eqref{eq:NAB} we have
$A=\bigl(1-\frac {u_{i+2}-1}{\gamma_{i+2}}\bigr)+f\bigl(1-\frac {u_{i+2}}{\gamma_{i+2}}\bigr)<\frac {2}{u_{i+2}}+\frac {f}{u_{i+2}}=\frac {f+2}{u_{i+2}}<\frac{\sqrtDelta+2}{u_{i+2}}$.
Because $n<(f+1)(r\tight+1)$, we also have
$B<(f+1)u_{i+2}\sqrtDelta+AN_i<(\sqrtDelta+1)u_{i+2}\sqrtDelta+\frac{\sqrtDelta+2}{u_{i+2}}\sqrtDelta.$
Finally,
\[
	N(\alpha) = AB
	< \frac{\sqrtDelta+2}{u_{i+2}} \bigl(\sqrtDelta+1\bigr)u_{i+2}\sqrtDelta+\bigl(\sqrtDelta+2\bigr)^{\!2}\sqrtDelta
	< \bigl(\sqrtDelta+2\bigr)\sqrtDelta\bigl(2\sqrtDelta+3\bigr)
.
\qedhere
\]
\end{proof}
The bound is essentially sharp in the first case discussed in the proof: 
For example, consider the situation when $i$ is odd, $u\eqdef u_{i+1}=u_{i+3}> \sqrtDelta/(1+\varepsilon)$ (for a constant $\varepsilon>0$), $u_{i+2}=1$, and $\alpha=u\alpha_i+u\alpha_{i+2}=u(2\alpha_i+\alpha_{i+1})$. 
As $\gamma_{i+2}>1$, using Lemma~\ref{lem:general-norm} we get that
\[
	N(\alpha)=u^2N(2\alpha_i+\alpha_{i+1})>u^2\sqrtDelta>\frac {\Delta^{3/2}}{(1+\varepsilon)^2}
.\]
There are infinitely many such examples for any $\varepsilon>0$, because there are infinitely many squarefree $D=n^2-1$, which then have $u_{2j}=2\lfloor\sqrt D\rfloor$ and $u_{2j+1}=1$ for all $j$. 
Examples with $\varepsilon\geq 4$, say, seem to be fairly common.

Note that one could also further refine the estimates in terms of sizes of the coefficients $u_j$, similarly as was done in~\cite{jang_kim_2016} for the norms of indecomposables.


\bibliographystyle{amsalpha}
\bibliography{biblio}

\end{document}